\documentclass[11pt]{amsart}
\usepackage[margin=1in]{geometry}
\usepackage{amsthm, amsmath,amsfonts,amssymb,euscript,hyperref,graphics,color,slashed,mathrsfs}
\usepackage{graphicx}
\usepackage{comment}
\usepackage{import}
\usepackage{tikz}
\usepackage{latexsym}
\usepackage{mathtools}

\newtheorem{theorem}{Theorem}[section]
\newtheorem{lemma}[theorem]{Lemma}

\newtheorem{remark}[theorem]{Remark}

\numberwithin{equation}{section}

\begin{document}
\title {At least two of $\zeta(5), \zeta(7), \ldots, \zeta(35)$ are irrational}

\author{Li Lai}
\address{Department of Mathematical Sciences, Tsinghua University \\ Beijing, China}
\email{lilaimath@gmail.com}

\author{Li Zhou}
\address{School of Mathematical Sciences, Fudan University \\ Shanghai, China}
\email{lzhou11@fudan.edu.cn}

\begin{abstract}
Let $\zeta(s)$ be the Riemann zeta function. We prove the statement in the title, which improves a recent result of Rivoal and Zudilin by lowering $69$ to $35$. We also prove that at least one of $\beta(2),\beta(4),\ldots,\beta(10)$ is irrational, where $\beta(s) = L(s,\chi_4)$ and $\chi_4$ is the Dirichlet character with conductor $4$.
\end{abstract}

\maketitle


\section{Introduction}
This paper mainly deals with the irrationality of values of the Riemann zeta function. It is well known (due to Euler and Lindemann) that $\zeta(2k)$ is transcendental for any positive integer $k$. A natural problem then is to investigate the irrationality of $\zeta(2k+1)$. More than four decades after Ap{\'e}ry's breakthrough that $\zeta(3) \notin \mathbb{Q}$ \cite{Ap79}, we still do not know any other single $\zeta(2k+1)$ to be irrational. (We refer the reader to the Bourbaki seminar notes \cite{Fi04} by Fischler in 2004 for a survey.) Nevertheless, there are some partial results after Ap{\'e}ry. In 2000 and 2001, Rivoal \cite{Ri00}, Ball and Rivoal \cite{BR01} showed that there are infinitely many positive integers $k$ such that $\zeta(2k+1)$ is irrational. During 2018 to 2020, starting with an idea of Zudilin, some new progress was made in \cite{FSZ19} and \cite{LY20}. Recently, Fischler \cite{Fi21} made a significant improvement on Ball-Rivoal's theorem.

Let us focus on $\zeta(2k+1)$ for small positive integers $k$. Rivoal \cite{Ri02} showed that at least one of $\zeta(5),\zeta(7),\ldots,\zeta(21)$ is irrational. In 2001, Zudilin achieved the following result. (See \cite{Zu04} for a detailed treatment.)
\begin{theorem}[Zudilin \cite{Zu01}]
At least one of $\zeta(5),\zeta(7),\zeta(9),\zeta(11)$ is irrational.
\end{theorem}
On a different aspect, as a companion result in \cite{BR01}, Ball and Rivoal showed that there exists an odd integer $i \in [5,169]$ such that $1,\zeta(3),\zeta(i)$ are linearly independent over $\mathbb{Q}$. In 2010, Fischler and Zudilin \cite{FZ10} lowered $169$ to $139$ by refining Nesterenko's linear independence criterion.

Recently, Rivoal and Zudilin \cite{RZ20} showed that there are at least two irrational numbers amongst $\zeta(5),\zeta(7),\ldots,\zeta(69)$.  We remark that the authors of \cite{RZ20} did not pursue the full strength of their method for simplicity. In this paper we take a different approach to prove the following result.

\begin{theorem}\label{maintheorem}
At least two of $\zeta(5),\zeta(7),\ldots,\zeta(35)$ are irrational.
\end{theorem}

We briefly describe the approaches in \cite{RZ20} and in this paper. We construct some concrete rational functions to obtain linear forms in $1$ and the Riemann zeta values. To rule out the unwanted $\zeta(3)$, there are two different methods. The first method is ``taking twice derivatives'', developed in \cite{Ri02,Zu01}; the second one is ``inserting rational roots'', initially proposed in \cite{Zu18} and developed by Sprang \cite{Sp18}. In \cite{RZ20}, the authors combined these two methods; in this paper, we use purely the second method in an elaborated way.

In Sections 2-5, we deal with the theoretical and computational parts of Theorem \ref{maintheorem}. In the last section, we consider a related problem for Catalan's constant $\beta(2)$.

\bigskip

\noindent\textbf{Acknowledgements.} L.L. wishes to thank Professor Wadim Zudilin for teaching him about the $\Phi_n$-factors. We would like to thank the referees for carefully reading the manuscript and for giving constructive comments which helped to improve the quality of the paper. In particular, we are greatly indebted to one of the referees for providing better parameters for Theorem \ref{secondtheorem}.

\section{Rational Functions and Linear Forms}

Our approach is a combination of the constructions in \cite{Zu04} and \cite{LY20} (in turn, the latter is based on \cite{FSZ19}).

For an odd integer $s \geqslant 5$, consider a collection $(m_1,m_2;\delta_1,\delta_2,\ldots,\delta_{s+1})$ of integral parameters satisfying the conditions $m_1,m_2 \geqslant 1$,
\begin{equation}\label{conditionfordelta}
0 \leqslant \delta_j < \frac{m_2}{2}, \text{~for all~}j=1,2,\ldots,s+1, \text{~and~} \sum_{j=1}^{s+1} \delta_j < \frac{(s-2)m_2 - 8m_1}{2}.
\end{equation}
Denote $\delta_{\min} = \min_{1 \leqslant j \leqslant s+1 } \delta_j$ and
\[ \mathcal{Z} = \left\{ 1,\frac{1}{2},\frac{1}{3},\frac{2}{3} \right\}. \]
For any positive even integer $n$, we define the rational function
\begin{align}\label{R_n(t)}
R_n(t) =& ~ 2^{2(2m_1+m_2)n} 3^{3(2m_1+m_2)n} \frac{\prod_{j=1}^{s+1} \left( (m_2 - 2\delta_j)n \right)!}{n!^{8m_1+3m_2}} (2t+m_2n) \notag \\
& \times (t-m_1n) \frac{\prod_{\theta \in \mathcal{Z}} (t-m_{1}n+\theta)_{(2m_1+m_2)n}}{(t)_{m_{2}n+1}\prod_{j=1}^{s+1}(t+\delta_{j}n)_{(m_2-2\delta_j)n+1}},
\end{align}
where $(y)_{k} = y(y+1)\cdots(y+k-1)$ is the rising factorial of length $k$. Notice that the numerator and denominator of $R_n(t)$ have a common factor $(t)_{m_{2}n+1}$, also, the condition $\sum_{j=1}^{s+1} \delta_j < \frac{(s-2)m_2 - 8m_1}{2}$ implies that $\deg R_n \leqslant -n \leqslant -2$. Hence, $R_n(t)$ has the unique partial-fraction decomposition
\begin{equation}\label{partial_fraction_decomposition}
R_n(t) = \sum_{i=1}^{s+1} \sum_{k = \delta_{\min}n}^{(m_2-\delta_{\min})n} \frac{a_{i,k}}{(t+k)^i}.
\end{equation}

For any $\theta \in \mathcal{Z}$ we define
\[ S_{n,\theta} = \sum_{t = 1}^{\infty} R_n(t+\theta).  \]
Recall the definition of the Hurwitz zeta values:
\[ \zeta(i,\alpha) = \sum_{t=0}^{\infty} \frac{1}{(t+\alpha)^i}, \]
where $i \geqslant 2$ is an integer and $\alpha > 0$ is a real number.

It is direct to check that $R_n(t)$ possesses the symmetry $R_n(t) = -R_n(-t-m_{2}n)$, so
\[ a_{i,k} =  (-1)^{i+1} a_{i,m_{2}n-k}. \]
Then a standard argument (see \cite[Lemma 1]{FSZ19}) implies that we can express $S_{n,\theta}$ as a linear form in $1$ and the Hurwitz zeta values with rational coefficients.

\begin{lemma}
For all $\theta \in \mathcal{Z}$, we have
\[ S_{n,\theta} = \rho_{0,\theta} + \sum_{3 \leqslant i \leqslant s \atop i \text{~odd~}} \rho_i \zeta(i,\theta), \]
where the rational coefficient
\begin{equation}\label{rho_i}
\rho_i = \sum_{k = \delta_{\min}n}^{(m_2-\delta_{\min})n} a_{i,k} \quad(\text{for}~3 \leqslant i \leqslant s,~i \text{~odd})
\end{equation}
does not depend on $\theta \in \mathcal{Z}$, and
\begin{equation}\label{rho_0}
\rho_{0, \theta}=-\sum_{k = \delta_{\min}n}^{(m_2-\delta_{\min})n} \sum_{\ell=0}^{k} \sum_{i=1}^{s+1} \frac{a_{i, k}}{(\ell+\theta)^{i}}.
\end{equation}
\end{lemma}

\bigskip

\section{Arithmetic of Coefficients}
We proceed to investigate the arithmetic properties of the coefficients $a_{i,k}$. Let $D_N$ be the least common multiple of $1,2,\ldots,N$.

\begin{lemma}\label{arithmetic_lemma_for_a_i_k}
Suppose that $n$ is even and $n > s^2$. We have
\[ \Phi_{n}^{-1} D_{(m_2-2\delta_{\min})n}^{s+1-i} a_{i,k} \in \mathbb{Z} \]
for $1 \leqslant i \leqslant s+1$ and $\delta_{\min}n \leqslant k \leqslant (m_2-\delta_{\min})n$, where the product over primes
\[ \Phi_n  = \prod_{\sqrt{3(2m_1+m_2)n} < p \leqslant (m_2-2\delta_{\min})n} p^{\nu_0(n/p)} \]
is defined through the $1$-periodic function
\[ \nu_0(x) = \min_{ y \in \mathbb{R} } \nu(x,y)  \]
and
\begin{align*}
\nu(x,y) = &~\left( \sum_{j=1}^{s+1} \left( \lfloor (m_2 - 2\delta_j)x \rfloor - \lfloor y-\delta_{j}x \rfloor - \lfloor (m_2-\delta_j)x - y \rfloor \right) \right) \\
&+  \lfloor 2m_{1}x + 2y \rfloor - \lfloor m_{1}x + y \rfloor + \lfloor 2(m_1+m_2)x - 2y \rfloor - \lfloor (m_1+m_2)x - y \rfloor    \\
&+  \lfloor 3m_{1}x+3y \rfloor + \lfloor 3(m_1+m_2)x - 3y \rfloor - \lfloor y \rfloor - \lfloor m_2x-y \rfloor - (8m_1+3m_2)\lfloor x \rfloor.
\end{align*}
\end{lemma}

\begin{proof}
We first split the function $R_n(t)$ into a product of some standard building blocks. Let
\[ H_j(t) =  \frac{((m_2-2\delta_j)n)!}{(t+\delta_jn)_{(m_2-2\delta_j)n+1}}, \quad j=1,2,\ldots,s+1,\]
and
\begin{align*}
G_{1/2}(t) &= 2^{(2m_1+m_2)n} \cdot 2^{(2m_1+m_2)n}\frac{(t-m_1n+1/2)_{(2m_1+m_2)n}}{n!^{2m_1+m_2}}, \\
G_{1/3}(t) &= 3^{(2m_1+m_2)n/2} \cdot 3^{(2m_1+m_2)n}\frac{(t-m_1n+1/3)_{(2m_1+m_2)n}}{n!^{2m_1+m_2}}, \\
G_{2/3}(t) &= 3^{(2m_1+m_2)n/2} \cdot 3^{(2m_1+m_2)n}\frac{(t-m_1n+2/3)_{(2m_1+m_2)n}}{n!^{2m_1+m_2}}, \\
G_{1}^{-}(t) &= \frac{(t-m_1n)_{m_1n}}{n!^{m_1}}, \\
G_1^{+}(t) &= \frac{(t+m_2n+1)_{m_1n}}{n!^{m_1}}.
\end{align*}
(Recall that $n$ is even, so $G_{1/3}(t)$ and $G_{2/3}(t)$ are polynomials in $t$ with rational coefficients.) Then we can rewrite $R_n(t)$ as (see \eqref{R_n(t)})
\begin{equation}\label{R=GGGHHH}
R_n(t) = (2t+m_2n)G_{1/2}(t)G_{1/3}(t)G_{2/3}(t)G_{1}^{-}(t)G_1^{+}(t)  \prod_{j=1}^{s+1} H_j(t).
\end{equation}
It is well known that (see \cite[Lemma 16]{Zu04})
\[ D_{(m_2-2\delta_{\min})n}^{\ell} \cdot \frac{1}{\ell!} \left( H_j(t)(t+k) \right)^{(\ell)} \big|_{t=-k} \in \mathbb{Z} \]
for any non-negative integer $\ell$, any integer $k$ such that $\delta_{\min}n \leqslant k \leqslant (m_2-\delta_{\min})n$, and any index $j$ with $1 \leqslant j \leqslant s+1$. It is also elementary to show that (see \cite[Propsition 3.2]{LY20}) for $F(t) = G_{1/2}(t), G_{1/3}(t), G_{2/3}(t), G_{1}^{-}(t)$ and $G_1^{+}(t)$,
\[ D_{n}^{\ell} \cdot \frac{1}{\ell!} F^{(\ell)}(-k) \in \mathbb{Z} \]
for any non-negative integer $\ell$ and any integer $k$. By applying the Leibniz rule, we derive that
\begin{equation}\label{weak_arithmetic_lemma}
D_{(m_2-2\delta_{\min})n}^{s+1-i} a_{i,k} = D_{(m_2-2\delta_{\min})n}^{s+1-i} \cdot \frac{1}{(s+1-i)!} \left( R_n(t)(t+k)^{s+1} \right)^{(s+1-i)} \big|_{t=-k} \in \mathbb{Z}
\end{equation}
for all $1 \leqslant i \leqslant s+1$ and $\delta_{\min}n \leqslant k \leqslant (m_2-\delta_{\min})n$.

Now, by \cite[Lemma 18]{Zu04}, for any prime $p > \sqrt{3(2m_1+m_2)n}$, any integer $k$ with $\delta_{\min}n \leqslant k \leqslant (m_2-\delta_{\min})n$, any non-negative integer $\ell$, and any index $j=1,2,\ldots,s+1$, there hold the following estimates for the $p$-adic orders:
\begin{equation}\label{ord_p_H}
\operatorname{ord}_{p}\left( \left( H_j(t)(t+k)  \right)^{(\ell)} \big|_{t=-k} \right) \geqslant -\ell + \left\lfloor \frac{(m_2 - 2\delta_j)n}{p} \right\rfloor - \left\lfloor \frac{k-\delta_{j}n}{p} \right\rfloor - \left\lfloor \frac{(m_2-\delta_j)n - k}{p} \right\rfloor.
\end{equation}
Define the polynomial $G(t) = G_{1/2}(t)G_{1/3}(t)G_{2/3}(t)G_{1}^{-}(t)G_1^{+}(t)$. For any integer $k$ with $\delta_{\min}n \leqslant k \leqslant (m_2-\delta_{\min})n$, we have
\[ G(-k) = (-1)^k \frac{(2m_1n+2k)!(2(m_1+m_2)n-2k)!}{(m_1n+k)!((m_1+m_2)n-k)!} \cdot \frac{(3m_1n+3k)!(3(m_1+m_2)n-3k)!}{k!(m_2n-k)!} \cdot \frac{1}{n!^{8m_1+3m_2}}. \]
So for any prime $p > \sqrt{3(2m_1+m_2)n}$ and any integer $k$ with $\delta_{\min}n \leqslant k \leqslant (m_2-\delta_{\min})n$, for $\ell = 0$ we have
\begin{align}\label{ord_p_G}
\operatorname{ord}_{p}\left( G^{(\ell)}(-k) \right) \geqslant ~& - \ell +  \left\lfloor \frac{2m_{1}n + 2k}{p} \right\rfloor - \left\lfloor \frac{m_{1}n + k}{p} \right\rfloor + \left\lfloor \frac{2(m_1+m_2)n - 2k}{p} \right\rfloor   \notag  \\
&- \left\lfloor \frac{(m_1+m_2)n - k}{p} \right\rfloor +  \left\lfloor \frac{3m_{1}n+3k}{p} \right\rfloor + \left\lfloor \frac{3(m_1+m_2)n - 3k}{p} \right\rfloor \notag \\
&- \left\lfloor \frac{k}{p} \right\rfloor - \left\lfloor \frac{m_2n-k}{p} \right\rfloor - (8m_1+3m_2)\left\lfloor \frac{n}{p} \right\rfloor .
\end{align}
Then applying inductive arguments to $G^{(\ell)}(t) = \left(G(t) \cdot \frac{G'(t)}{G(t)} \right)^{(\ell-1)}$, similar to that in the proof of \cite[Lemma 17]{Zu04}, we deduce that the inequality \eqref{ord_p_G} holds for any non-negative integer $\ell$.

Finally, we write \eqref{R=GGGHHH} as $R_n(t) = (2t+m_2n) G(t) \prod_{j=1}^{s+1} H_j(t)$. Applying the Leibniz rule again with the help of \eqref{ord_p_H} and \eqref{ord_p_G}, we obtain that
\begin{align}\label{saving_large_primes}
\operatorname{ord}_{p} (a_{i,k}) &= \operatorname{ord}_{p} \left(  \frac{1}{(s+1-i)!} \left( R_n(t)(t+k)^{s+1} \right)^{(s+1-i)} \big|_{t=-k} \right) \notag \\
&\geqslant -(s+1-i) + \nu\left( \frac{n}{p},\frac{k}{p} \right) \notag \\
&\geqslant -(s+1-i) + \nu_0(n/p)
\end{align}
for any prime $p > \max\{ \sqrt{3(2m_1+m_2)n},s \}$, any $i=1,2,\ldots,s+1$, and any integer $k$ such that $\delta_{\min}n \leqslant k \leqslant (m_2-\delta_{\min})n$. Combining \eqref{weak_arithmetic_lemma} and \eqref{saving_large_primes}, we complete the proof of Lemma \ref{arithmetic_lemma_for_a_i_k}.
\end{proof}

\bigskip

We study the coefficients $\rho_i$ and $\rho_{0,\theta}$ in the following lemma.
\begin{lemma}\label{arithmeticlemma}
Let $n$ be an even integer and $n > s^2$. We have~$\Phi_{n}^{-1} D_{(m_2-2\delta_{\min})n}^{s+1-i}  \rho_i \in \mathbb{Z}$ for any odd $i$, $3 \leqslant i \leqslant s$. Moreover, we have
\[ \Phi_{n}^{-1} D_{(m_2-2\delta_{\min})n}^{s+1} \rho_{0,\theta} \in \mathbb{Z} \text {~for any~} \theta \in \mathcal{Z}\setminus\{1\},\]
and
\[\left( \Phi_{n}^{-1}  \prod_{j=1}^{s+1} D_{\max\{ (m_2-2\delta_{\min})n, (m_2 - \delta_j)n \}+1} \right) \rho_{0,1} \in \mathbb{Z}.  \]
\end{lemma}

\begin{proof}
Recall the definition \eqref{rho_i} of $\rho_i$, the first assertion follows immediately from Lemma \ref{arithmetic_lemma_for_a_i_k}.

In the following, we assume without loss of generality that $\delta_1 \leqslant \delta_2 \leqslant \cdots \leqslant \delta_{s+1}$. Then
\[ a_{i,k} = 0 \text{~if~} k > (m_2- \delta_{i})n. \]
(Because the order of pole of $R_n(t)$ at $t=-k$ is at most $i-1$ when $k > (m_2- \delta_{i})n$). Denote $M_j = \max\{ (m_2-2\delta_{\min})n, (m_2 - \delta_j)n \}+1$ for $j=1,2,\ldots,s+1$, and $M = (m_2-2\delta_{\min})n$. So $(m_2 - \delta_{\min})n+1 \geqslant M_1 \geqslant M_2 \geqslant \cdots \geqslant M_{s+1} \geqslant M+1$.

We now prove the last assertion. Since $a_{i,k} = 0$ when $k \geqslant M_i$, we can write $\rho_{0,1}$ defined in \eqref{rho_0} as
\begin{equation}\label{rho_0_preciser_form}
\rho_{0, 1}=- \sum_{i=1}^{s+1} \sum_{k = \delta_{\min}n}^{M_{i}-1} \left( a_{i,k} \sum_{\ell=0}^{k}  \frac{1}{(\ell+1)^{i}} \right).
\end{equation}
By Lemma \ref{arithmetic_lemma_for_a_i_k} and $M_j \geqslant M$, we know that
\[ \left( \Phi_{n}^{-1}  \prod_{j=i+1}^{s+1} D_{M_j} \right) a_{i,k} \in \mathbb{Z} \]
for all $1 \leqslant i \leqslant s+1$ and $\delta_{\min}n \leqslant k \leqslant M_{i}-1$. Clearly,
\[ \prod_{j=1}^{i} D_{M_j} \cdot  \sum_{\ell=0}^{k}  \frac{1}{(\ell+1)^{i}} \in \mathbb{Z} \]
for all $1 \leqslant i \leqslant s+1$ and $\delta_{\min}n \leqslant k \leqslant M_{i}-1$. Therefore, from \eqref{rho_0_preciser_form} we see that the last assertion holds.

To prove the second assertion, we argue by contradiction. Suppose that $ \Phi_{n}^{-1} D_{M}^{s+1} \rho_{0,\theta} \notin \mathbb{Z}$ for some $\theta \in \mathcal{Z} \setminus\{1\}$, then by \eqref{rho_0}, there exist $k_0$, $\ell_0$ such that $\delta_{\min}n \leqslant k_0 \leqslant (m_2 - \delta_{\min})n$, $0 \leqslant \ell_0 \leqslant k_0$, and
\[ \Phi_{n}^{-1} D_{M}^{s+1} \cdot \sum_{i=1}^{s+1} \frac{a_{i,k_0}}{(\ell_0+\theta)^{i}} \notin \mathbb{Z}. \]
Note that $R_n(\ell_0-k_0+\theta) = 0$, so by \eqref{partial_fraction_decomposition}, we have
\[ \Phi_{n}^{-1} D_{M}^{s+1} \cdot \sum_{i=1}^{s+1} \frac{a_{i,k_0}}{(\ell_0+\theta)^{i}} = - \Phi_{n}^{-1} D_{M}^{s+1} \cdot \sum_{k=\delta_{\min}n \atop k \neq k_0}^{(m_2 - \delta_{\min})n} \sum_{i=1}^{s+1} \frac{a_{i,k}}{(\ell_0-k_0+k+\theta)^{i}} \notin \mathbb{Z}. \]
Thus, there exist a prime $p$, some $i_0,i_1 \in \{ 1,\ldots,s+1 \} $, and some $k_1 \in \{ \delta_{\min}n,\ldots,  (m_2 - \delta_{\min})n \}$ with $k_1 \neq k_0$ such that
\[ \operatorname{ord}_p\left( \Phi_{n}^{-1} D_{M}^{s+1} \cdot \frac{a_{i_0,k_0}}{(\ell_0+\theta)^{i_0}}  \right) < 0, \quad \operatorname{ord}_p\left( \Phi_{n}^{-1} D_{M}^{s+1} \cdot \frac{a_{i_1,k_1}}{(\ell_0-k_0+k_1+\theta)^{i_1}}  \right) < 0. \]
Since $\Phi_{n}^{-1} D_{M}^{s+1-i} a_{i,k} \in \mathbb{Z}$ for all $i,k$ by Lemma \ref{arithmetic_lemma_for_a_i_k}, we deduce that
\[ \operatorname{ord}_p\left( \ell_0+\theta  \right) > \operatorname{ord}_p\left( D_M \right), \quad \operatorname{ord}_p\left( \ell_0-k_0+k_1+\theta  \right) > \operatorname{ord}_p\left( D_M \right). \]
Hence,
\[  \operatorname{ord}_p\left( |k_0-k_1| \right) > \operatorname{ord}_p\left( D_M \right), \]
but it contradicts the fact that $0 < |k_0-k_1| \leqslant (m_2 - 2\delta_{\min})n = M$. This completes the proof of Lemma \ref{arithmeticlemma}.
\end{proof}

\bigskip

\section{Asymptotics and Proof of Theorem \ref{maintheorem} }\label{Section_4}

The asymptotics of $\Phi_n$ can be established easily by the prime number theorem ($\sum_{p \leqslant x} \log p \sim x$ as $x \rightarrow +\infty$) in the following lemma. Such lemmas are sometimes called Chudnovsky-Rukhadze-Hata arguments. For details, we refer the reader to \cite[p. 341]{Ha93} and \cite[Lemma 4.4]{Zu02}.
\begin{lemma}\label{Phi}
We have
\[
\lim_{n \rightarrow +\infty} \frac{\log \Phi_n}{n} = \int_{0}^{1} \nu_0(x) {\rm d} \psi(x) + \int_{0}^{\frac{1}{m_2-2\delta_{\min}}} \nu_0(x) {\rm d} \left( \frac{1}{x} \right),
\]
where $\psi(x) = \frac{\Gamma'(x)}{\Gamma(x)}$ is the digamma function.
\end{lemma}

\bigskip

Now, we study the asymptotics of $S_{n,\theta}$. The following lemma is a modification of \cite[Lemma~4.1]{LY20}, it only involves Stirling's formula.

\begin{lemma}\label{analysislemma}
We have
\[ \lim_{n \rightarrow +\infty} S_{n,1}^{1/n} = g(x_0), \]
where the function $g$ is defined by
\begin{align*}
g(X) = &~108^{2m_1+m_2} \left( \prod_{j=1}^{s+1} \left( m_2 - 2\delta_j \right)^{m_2 - 2\delta_j} \right) \left(2m_1+m_2+X\right)^{4(2m_1+m_2)}  \\
& \times \frac{(m_1+X)^{m_1}}{(m_1+m_2+X)^{m_1+m_2}} \prod_{j=1}^{s+1} \frac{(m_1+\delta_j+X)^{m_1+\delta_j}}{(m_1+m_2-\delta_j+X)^{m_1+m_2-\delta_j}},
\end{align*}
and $x_0$ is the unique positive real solution of $f(X) = 1$ with the function $f$ defined by
\[ f(X) = \left( \frac{2m_1+m_2+X}{X} \right)^4 \frac{m_1+X}{m_1+m_2+X} \prod_{j=1}^{s+1} \frac{m_1+\delta_j+X}{m_1+m_2-\delta_j+X}. \]
Moreover, for any $\theta \in \mathcal{Z}$, we have
\[ \lim_{n \rightarrow +\infty} \frac{S_{n,1}}{S_{n,\theta}} = 1. \]
\end{lemma}

\begin{proof}
Firstly, we show that $f'(x) = 0$ has a unique solution $x=x_1$ in $(0,+\infty)$. By computing the log-derivative of $f(x)$, we obtain that
\begin{align*}
\frac{f'(x)}{f(x)} &= -\frac{4(2m_1+m_2)}{x(2m_1+m_2+x)} + \frac{m_2}{(m_1+x)(m_1+m_2+x)} + \sum_{j=1}^{s+1}\frac{m_2-2\delta_j}{(m_1+\delta_j+x)(m_1+m_2-\delta_j+x)} \\
&= u(x)/(x(2m_1+m_2+x)),
\end{align*}
where
\begin{align*}
u(x) = ~&-4(2m_1+m_2)+m_2\left( 1-\frac{m_1(m_1+m_2)}{(m_1+x)(m_1+m_2+x)} \right) \\
&+ \sum_{j=1}^{s+1} (m_2-2\delta_j)\left( 1- \frac{(m_1+\delta_j)(m_1+m_2-\delta_j)}{(m_1+\delta_j+x)(m_1+m_2-\delta_j+x)} \right).
\end{align*}
Clearly $u(x)$ is increasing on $(0,+\infty)$. Since $u(0^{+}) = -4(2m_1+m_2)<0$ and $u(+\infty) = (s-2)m_2-8m_1 -2\sum_{j=1}^{s+1} \delta_j > 0$ (by \eqref{conditionfordelta}), there is a unique $x_1 \in (0,+\infty)$ such that $u(x_1) = 0$.

Therefore, $f(x)$ is decreasing on $(0,x_1)$ and increasing on $(x_1,+\infty)$. Since $f(0^{+}) = +\infty$ and $f(+\infty) = 1$, we see that there exists a unique $x_0 \in (0,x_1)$ such that $f(x_0) = 1$. Moreover, $f(x) > 1$ for $x \in (0,x_0)$ and $f(x)<1$ for $x \in (x_0,+\infty)$.

The remaining proof works in the same way as in \cite[Lemma~4.1]{LY20}. We only sketch the main steps as follows. Since $R_n(k+\theta) = 0$ for all $k=1,2,\ldots,m_{1}n-1$ and any $\theta \in \mathcal{Z}$, we can write
\[ S_{n,\theta} =  \sum_{k=0}^{\infty} R_n(m_1n+k+\theta). \]
(Each term in the above summation is positive.) Suppose for the moment that $k = \kappa n$ for some constant $\kappa > 0$, then by Stirling's formula in the form $\Gamma(x)=x^{O_{x \rightarrow+\infty}(1)}\left(\frac{x}{e}\right)^{x}$ we derive that, as $n \rightarrow +\infty$,
\begin{align*}
R_n\left( (m_1+\kappa)n+\theta \right)^{1/n} =~& n^{O(1/n)} 108^{2m_1+m_2} \prod_{j=1}^{s+1} (m_2-2\delta_j)^{m_2-2\delta_j} \\
&\times \left( \left( \frac{2m_1+m_2+\kappa}{\kappa} \right)\frac{m_1+\kappa}{m_1+m_2+\kappa} \prod_{j=1}^{s+1} \frac{m_1+\delta_j+\kappa}{m_1+m_2-\delta_j+\kappa}  \right)^{\kappa} \\
&\times (2m_1+m_2+\kappa)^{4(2m_1+m_2)}\frac{(m_1+\kappa)^{m_1}}{(m_1+m_2+\kappa)^{m_1+m_2}} \\
&\times \prod_{j=1}^{s+1}\frac{(m_1+\delta_j+\kappa)^{m_1+\delta_j}}{(m_1+m_2-\delta_j+\kappa)^{m_1+m_2-\delta_j}} \\
=~&n^{O(1/n)}f(\kappa)^{\kappa}g(\kappa).
\end{align*}
Define the function $h(x) = f(x)^{x}g(x)$ on $(0,+\infty)$, then a direct computation gives $\frac{h'(x)}{h(x)} = \log f(x) + x\frac{f'(x)}{f(x)} + \frac{g'(x)}{g(x)} = \log f(x)$. So $h(x)$ is increasing on $(0,x_0)$, decreasing on $(x_0,+\infty)$ and $h(x_0)=g(x_0)$. Based on the above observation, we can show that $S_{n,\theta} = n^{O(1)}h(x_0)^n$ so that $\lim_{n \rightarrow +\infty} S_{n,1}^{1/n} = h(x_0) = g(x_0)$. Moreover, for any prescribed sufficiently small $\varepsilon_0 > 0$ and $\theta \in \mathcal{Z}$, it can be shown that
\begin{align*}
S_{n,1} &= (1+o(1))\sum_{(x_0-\varepsilon_0)n \leqslant k \leqslant (x_0+\varepsilon_0)n} R_n(m_1n+k+1),  \\
S_{n,\theta} &= (1+o(1))\sum_{(x_0-\varepsilon_0)n \leqslant k \leqslant (x_0+\varepsilon_0)n} R_n(m_1n+k+\theta),
\end{align*}
as $n \rightarrow +\infty$. Then by using $\frac{\Gamma(x+1-\theta)}{\Gamma(x)}=\left(1+o_{x \rightarrow+\infty}(1)\right) x^{1-\theta}$ (which is a corollary of Stirling's formula), we obtain uniformly for $(x_0-\varepsilon_0)n \leqslant k \leqslant (x_0+\varepsilon_0)n$ that, as $n \rightarrow +\infty$,
\[ \frac{R_n(m_1n+k+1)}{R_n(m_1n+k+\theta)} = (1+o(1))f\left( k/n \right)^{1-\theta}. \]
Therefore,
\[ f(x_0+\varepsilon_0)^{1-\theta} \leqslant \liminf_{n \rightarrow +\infty} \frac{S_{n,1}}{S_{n,\theta}} \leqslant \limsup_{n \rightarrow +\infty} \frac{S_{n,1}}{S_{n,\theta}} \leqslant f(x_0-\varepsilon_0)^{1-\theta}. \]
By letting $\varepsilon_0 \rightarrow 0^{+}$ we finally obtain that $\lim_{n \rightarrow +\infty} \frac{S_{n,1}}{S_{n,\theta}} = 1$.
\end{proof}

\bigskip

Now, we prove Theorem \ref{maintheorem}.
\begin{proof}[Proof of Theorem \ref{maintheorem}]
Since $\zeta(3)$ is irrational, Theorem \ref{maintheorem} is equivalent to the following assertion: there are at least three numbers among $\zeta(3),\zeta(5),\ldots,\zeta(35)$ that are irrational.

Take $s = 35$. Suppose that there were only two odd integers $i_1 = 3$ (by Ap{\' e}ry) and $i_2 \in \{ 5,7,9,11 \}$ (by Zudilin \cite{Zu01}) such that $\zeta(i_1)$ and $\zeta(i_2)$ are irrational, and for all $i \in \{ 3,5,\ldots,35 \} \setminus\{ i_1,i_2\}$, $\zeta(i)$ is rational; let $A$ be the common denominator of these rational $\zeta(i)$. Since the generalized Vandermonde matrix
\[\left(\begin{array}{rrr}
1 & 2~ & 3~ \\
1 & 2^{i_1} & 3^{i_1} \\
1 & 2^{i_2} & 3^{i_2}
\end{array}\right)
\]
is invertible, there exist $w_1,w_2,w_3 \in \mathbb{Z}$ such that $w_1+2^{i_1}w_2 + 3^{i_1}w_3 = w_1+2^{i_2}w_2+3^{i_2}w_3 = 0$ and $w_1+2w_2+3w_3 \neq 0$. Since
\[ \sum_{k=1}^{b} \zeta\left(i, \frac{k}{b}\right)=\sum_{k=1}^{b} \sum_{m=0}^{\infty} \frac{b^{i}}{(m b+k)^{i}}=b^{i} \zeta(i), \]
we derive that for any $b \in \{ 1,2,3 \}$,
\[ \widehat{S}_{n, b}:=\sum_{k=1}^{b} S_{n, k / b}=\sum_{k=1}^{b} \rho_{0, k / b}+\sum_{i \in \{ 3,5,\ldots,35 \}} \rho_{i} b^{i} \zeta(i) \]
is a linear combination of $1$ and the Riemann zeta values. By Lemma \ref{analysislemma}, we have $\widehat{S}_{n, b} = (b+o(1))S_{n,1}$ as $n \rightarrow +\infty$. Let
\[ \widetilde{S}_{n}:=\sum_{b=1}^{3} w_{b} \widehat{S}_{n, b}; \]
then
\[ \widetilde{S}_{n}=\sum_{b=1}^{3} w_{b} \sum_{k=1}^{b} \rho_{0, k / b}+\sum_{i \in \{ 3,5,\ldots,35 \} \setminus\{ i_1,i_2\}}\left(\sum_{b=1}^{3} w_{b} b^{i}\right) \rho_{i} \zeta(i) \]
and
\begin{equation}\label{widetilde_S_n}
\widetilde{S}_{n}= (w_1+2w_2+3w_3+o(1))S_{n,1} \quad\text{~with~} w_1+2w_2+3w_3 \neq 0.
\end{equation}
By Lemma \ref{arithmeticlemma}, we have
\begin{equation}\label{APhiDSinZ}
\left( A\Phi_{n}^{-1}  \prod_{j=1}^{s+1} D_{\max\{ (m_2-2\delta_{\min})n, (m_2 - \delta_j)n \}+1} \right) \widetilde{S}_{n} \in \mathbb{Z}
\end{equation}
for any positive even integer $n>s^2$.

On the other hand, by Lemma \ref{Phi}, we have $\lim_{n\rightarrow +\infty}  \Phi_{n}^{1/n} = \exp(C_1)$, where
\begin{equation}\label{C_1}
C_1 :=  \int_{0}^{1} \nu_0(x) {\rm d} \psi(x) + \int_{0}^{\frac{1}{m_2-2\delta_{\min}}} \nu_0(x) {\rm d} \left( \frac{1}{x} \right).
\end{equation}
The prime number theorem ($D_{N} = \exp((1+o_{N \rightarrow +\infty}(1))N)$) implies that
\begin{equation}\label{PNT}
\lim_{n\rightarrow +\infty} \left(\prod_{j=1}^{s+1} D_{\max\{ (m_2-2\delta_{\min})n, (m_2 - \delta_j)n \}+1} \right)^{1/n} = \exp\left(\sum_{j=1}^{s+1} \max\{ m_2- 2\delta_{\min},m_2 - \delta_j \}\right).
\end{equation}
By \eqref{widetilde_S_n} and Lemma \ref{analysislemma}, we have
\begin{equation}\label{C_2_2}
\lim_{n\rightarrow +\infty} \widetilde{S}_{n}^{1/n} = g(x_0).
\end{equation}
Putting \eqref{C_1}, \eqref{PNT} and \eqref{C_2_2} together, we obtain that
\begin{equation}\label{Exp_C_2_C_1}
\lim_{n \rightarrow +\infty} \left(\left( A\Phi_{n}^{-1}  \prod_{j=1}^{s+1} D_{\max\{ (m_2-2\delta_{\min})n, (m_2 - \delta_j)n \}+1} \right) \widetilde{S}_{n} \right)^{1/n} = e^{-C_1+C_2},
\end{equation}
where $C_2$ is defined as follows:
\begin{align*}
C_2 := &~ \left(\sum_{j=1}^{s+1} \max\{ m_2- 2\delta_{\min},m_2 - \delta_j \} \right) + (2m_1+m_2)\log 108 + \left(\sum_{j=1}^{s+1} (m_2-2\delta_j) \log (m_2-2\delta_j) \right) \\
&+ 4(2m_1+m_2)\log (2m_1+m_2+x_0) + m_1 \log (m_1+x_0) - (m_1+m_2) \log(m_1+m_2+x_0) \\
&+ \left( \sum_{j=1}^{s+1} \left( (m_1+\delta_j)\log(m_1+\delta_j+x_0) - (m_1+m_2-\delta_j) \log (m_1+m_2-\delta_j+x_0) \right) \right).
\end{align*}

If $C_1 > C_2$, we will obtain a contradiction of \eqref{APhiDSinZ} and \eqref{Exp_C_2_C_1}. Take the parameters as follows: $s=35$; $m_1=209$, $m_2=243$; $\delta_j =4$ for $1 \leqslant j \leqslant 5$; $\delta_j = j-1$ for $6 \leqslant j \leqslant 11$; $\delta_j = 2j-12$ for $12 \leqslant j \leqslant 32$; and $\delta_{j} = 4j-76$ for $33 \leqslant j \leqslant 36$. By a MATLAB program, we find that $x_0 = 2.89493833\ldots$ and
\[ C_1 = 16779.9312\ldots > C_2 = 16779.2826\ldots.\]
This contradiction completes the proof of Theorem \ref{maintheorem}.
\end{proof}

We will describe the MATLAB code and give a website link to it in the next section. The above parameters are found by random search and trial-and-error.

We conclude this section by some remarks about Theorem \ref{maintheorem}.

\begin{remark}
If one elaborates the method in \cite{RZ20}, some first attempts suggest that one cannot obtain a result better than Theorem \ref{maintheorem}. However, we did not put our effort on figuring it out.
\end{remark}

\begin{remark}
It is possible that the arithmetic behavior of $\rho_i$ and $\rho_{0,\theta}$ is even better, by considering certain hypergeometric transformations underlying the construction. See the explanation of the ``denominator conjecture'' in \cite[Chapitre 17]{KR07}. It is tremendously difficult to put such things into consideration in this paper.
\end{remark}

\begin{remark}\label{rmk}
We have some other choices for the denominator factor $n!^{8m_1+3m_2}$ of $R_n(t)$, due to different arithmetic normalization of the building blocks $G_{1/2}(t),G_{1/3}(t),G_{2/3}(t),G_{1}^{-}(t)$ and $G_1^{+}(t)$ in the proof of Lemma \ref{arithmetic_lemma_for_a_i_k}. Recall that $G_{1/2}(t)$ is the product of the following $2m_1+m_2$ polynomials:
\begin{align*}
&\frac{4^{n}(t-m_{1}n+1/2)_n}{n!}, \frac{4^{n}(t-m_{1}n+n+1/2)_n}{n!}, \frac{4^{n}(t-m_{1}n+2n+1/2)_n}{n!},\ldots, \\
&\frac{4^{n}(t+(m_1+m_2-1)n+1/2)_n}{n!}.
\end{align*}
In general, we can replace $G_{1/2}(t)$ by the product of
\begin{align}\label{replacement}
&\frac{4^{u_{1}n}(t-m_{1}n+1/2)_{u_{1}n}}{(u_{1}n)!}, \frac{4^{u_{2}n}(t-m_{1}n+u_{1}n+1/2)_{u_{2}n}}{(u_{2}n)!},\frac{4^{u_{3}n}(t-m_{1}n+u_{1}n+u_{2}n+1/2)_{u_{3}n}}{(u_{3}n)!},\ldots, \notag \\
&\frac{4^{u_{I}n}(t-m_{1}n+u_{1}n+u_{2}n+\cdots+u_{I-1}n+1/2)_{u_{I}n}}{(u_{I}n)!},
\end{align}
where $u_1,u_2,\ldots,u_I$ are arbitrary positive integers satisfying the conditions
\begin{equation}\label{conditionsforbuildingblocks}
\sum_{i=1}^{I} u_i = \frac{1}{n}\deg G_{1/2}(t) =(2m_1+m_2) \quad \text{and} \quad \max_{1 \leqslant i \leqslant I} u_i \leqslant m_2-2\delta_{\min}. 
\end{equation}
We can replace $G_{1/3}(t)$, $G_{2/3}(t)$, $G_{1}^{-}(t)$ and $G_{1}^{+}(t)$ in a similar way.

Such replacements will not affect the finial result of Theorem \ref{maintheorem}. We explain it through a simple example below, and there is no difficulty for the general case. Suppose that we take $u_1=2$ and $u_2=\cdots=u_{I}=1$ in \eqref{replacement} for $G_{1/2}(t)$; namely, $G_{1/2}(t)$ is replaced by $\widetilde{G}_{1/2}(t) = \frac{n!^2}{(2n)!}G_{1/2}(t)$ (so does $R_n(t)$).
Let
\[ F(t) = \frac{4^{2n}(t-m_{1}n+1/2)_{2n}}{(2n)!}.  \]
As in the proof of Lemma \ref{arithmetic_lemma_for_a_i_k}, we have
\[ D_{2n}^{\ell} \cdot \frac{1}{\ell!} F^{(\ell)}(-k) \in \mathbb{Z}, \]
so we still have  
\begin{equation}\label{weak_arithmetic_lemma_in_remark}
D_{(m_2-2\delta_{\min})n}^{s+1-i} a_{i,k} \in \mathbb{Z},
\end{equation}
provided that $2 \leqslant m_2-2\delta_{\min}$. (In the general case, the latter condition in \eqref{conditionsforbuildingblocks} is used to insure \eqref{weak_arithmetic_lemma_in_remark}.) The factor $\Phi_n$ and the functions $\nu(x,y),\nu_0(x,y),g(x)$ change slightly according to the replacement. Eventually, $C_2$ becomes $\widetilde{C}_2 = C_{2} - 2\log 2$ because of
\[ \lim_{n \rightarrow +\infty} \left( \frac{(2n)!}{n!^2} \right)^{1/n} = 2^2. \]
Meanwhile, $\Phi_n$ becomes
\[ \widetilde{\Phi}_n = \Phi_n \Bigg/ \prod_{p > \sqrt{3(2m_1+m_2)n}} p^{\operatorname{ord}_{p}((2n)!/n!^2)}.\]
Note that 
\[ \prod_{p \leqslant \sqrt{3(2m_1+m_2)n}} p^{\operatorname{ord}_{p}((2n)!/n!^2)} = \exp\left( O(\pi(\sqrt{3(2m_1+m_2)n}) \cdot \log n) \right)= \exp(O(\sqrt{n}))  \]
is negligible, so $C_1 = \lim_{n \rightarrow +\infty} (\log \Phi_n)/n$ becomes $\widetilde{C}_1 = C_1 - 2\log 2$. Thus, $-C_1+C_2$ remains unchanged, so does Theorem \ref{maintheorem}.

\end{remark}

\section{Computational Aspect}\label{Computation_Section}
In this section, we explain how to calculate $C_1$ numerically. (The calculations for $x_0$ and $C_2$ are straightforward.)

The function $\nu(x,y)$ (defined in Lemma \ref{arithmetic_lemma_for_a_i_k}) can be rewritten as
\[ \nu(x,y) = \sum_{i=1}^{H}  h_{i,1}\lfloor h_{i,2}x+h_{i,3}y \rfloor , \]
where $h$ is the $H \times 3$ matrix with integral entries:
\[h = \left(\begin{array}{rrr}
	1 & m_2 - 2\delta_{1} & 0 \\
	-1 & -\delta_{1} & 1 \\
	-1 & m_2-\delta_{1} & -1\\
	 \ldots&\ldots&\ldots\\
	 1 & m_2 - 2\delta_{j} & 0 \\
	 -1 & -\delta_{j} & 1 \\
	 -1 & m_2-\delta_{j} & -1\\
	 \ldots&\ldots&\ldots\\
    1 & m_2 - 2\delta_{s+1} & 0 \\
	-1 & -\delta_{s+1} & 1 \\
	-1 & m_2-\delta_{s+1} & -1\\
	 1 & 2m_{1} & 2\\
	 -1 & m_{1} & 1\\
	1 & 2(m_1+m_2) & -2\\
	-1 & (m_1+m_2) & -1\\
	1 & 3m_{1} & 3\\
	1 & 3(m_1+m_2) & -3\\
	-1 & 0 & 1\\
	-1 & m_{2} & -1\\
	-(8m_1 + 3m_2)&1 &0\\
\end{array}\right).
\]
Clearly, $\nu(x,y)$ is $1$-periodic in both variables $x$ and $y$. In the $xy$-plane, all the lines $h_{i,2}x + h_{i,3}y = k$, ($i=1,\ldots,H$, $k \in \mathbb{Z}$), cut apart the $xy$-plane into polygons, and $\nu(x,y)$ is constant in the interior of each polygon. Moreover, the value of $\nu(x,y)$ at a non-vertex point on the common side of two polygons is equal to the $\nu$-value of the interior points of one of these two polygons. Thus the function $\nu_0(x)=\min_{y \in \mathbb{R}} \nu(x,y)$ is $1$-periodic, and any of its discontinuities must be the $x$-coordinate of the intersection point of some pair of lines.

Let $X=\bigcup_{q_{i,j}\neq 0}X_{i,j}$ where $X_{i,j}=\{\frac{k}{q_{i,j}} \mid k\in\mathbb{Z}, \ 0 \leqslant k \leqslant q_{i,j} \}$ and $q_{i,j}$ is defined as
\[
q_{i,j} = \operatorname{abs} \left(
\left|\begin{array}{cc}
    h_{i,2}   & h_{i,3}\\
    h_{j,2}   & h_{j,3}\\
\end{array}\right| \right).
\]
Suppose $X=\{0=x_0<x_1<x_2<\cdots<x_l=1\}$, then $X$ contains all the discontinuities of $\nu_0(x)$ in the interval $[0,1]$, and there exists an index $l_{\text{mid}}$ such that $x_{l_{\text{mid}}}=\frac{1}{m_2-2\delta_{\min}} \in X$. Then $\nu_0(x)$ is constant on each interval $(x_{i-1},x_{i})$, and we can express the integration \eqref{C_1} as a finite summation:
\[ C_1 = \sum_{i=1}^{l} \nu_0\left(\frac{x_{i-1}+x_{i}}{2}\right)\cdot(\psi(x_{i})-\psi(x_{i-1})) +
\sum_{i=1}^{l_{\text{mid}}} \nu_0\left(\frac{x_{i-1}+x_{i}}{2}\right)\cdot\left(\frac{1}{x_{i}}-\frac{1}{x_{i-1}}\right). \]
(To reduce the effect of round off errors, we choose to calculate $\nu_0\left(\frac{x_{i-1}+x_{i}}{2}\right)$ on each interval $(x_{i-1},x_i)$.)

For a fixed $\hat{x} = \frac{x_{i-1}+x_{i}}{2}$, the function $y \mapsto \nu(\hat{x},y)$ is piecewise constant and any of its discontinuities must be the $y$-coordinate of the intersection point of the line $x = \hat{x}$ and some line in the form $h_{j,2}x + h_{j,3}y = k$. Let $Y = Y(\hat{x}) = \bigcup_{h_{j,3}\neq 0}Y_{j}$, where
\[
Y_j=\left\{\frac{k-h_{j,2}\hat{x}}{h_{j,3}} ~\Bigg|~ k \in \mathbb{Z},~
\min\{h_{j,2}\hat{x},h_{j,2}\hat{x}+h_{j,3}\} \leqslant k \leqslant \max\{h_{j,2}\hat{x},h_{j,2}\hat{x}+h_{j,3}\}
 \right\},
\]
then $Y$ contains all the discontinuities of the function $y \mapsto \nu(\hat{x},y)$ in $[0,1]$. Let $Y'$ be the set of middle points of two consecutive numbers in $Y$, then $\nu_0(\hat{x}) = \min_{y \in Y'} \nu(\hat{x},y)$ is the minimum of finitely many terms. In this way, we can calculate the value of $C_1$.

The MATLAB code \textsf{zeta35.m} can be downloaded at \url{https://github.com/lzhou-xyz/zeta35/}. It takes around one minute on a personal laptop to obtain the result.

\section{A related problem about Catalan's constant}
It is natural to generalize the results about the Riemann zeta values to the Dirichlet $L$-values. We refer the reader to Fischler \cite{Fi20} for recent progress. In the following, we consider the Dirichlet beta function; that is,
\[ \beta(s) = L(s,\chi_4) = \sum_{j=0}^{\infty} \frac{(-1)^j}{(2j+1)^s},~\operatorname{Re}(s) > 0.  \]
As in the Riemann zeta case, a half of the $\beta$-values at positive integers are ``trivially'' transcendental: Euler showed that $\beta(2k+1)$ is a non-zero rational multiple of $\pi^{2k+1}$ for any non-negative integer $k$, so $\beta(2k+1)$ is transcendental. We know little about $\beta(2k)$.

The constant $\beta(2)$ is called Catalan's constant. Unlike Ap{\'e}ry's constant $\zeta(3)$, we still do not know whether $\beta(2)$ is irrational or not. In 2003, Rivoal and Zudilin \cite{RZ03} showed that at least one of $\beta(2),\beta(4),\ldots,\beta(14)$ is irrational. Recently, Zudilin \cite{Zu19} improved $14$ to $12$. By using the same constructions in \cite{Zu19} with just a different collection of parameters, we find that it can be improved further.

\begin{theorem}\label{secondtheorem}
At least one of $\beta(2),\beta(4),\beta(6),\beta(8),\beta(10)$ is irrational.
\end{theorem}

The proof is identical to \cite[\S 3]{Zu19} except for small modifications: we need to take a different normalization factor of the rational function. In the following, we repeat the process of \cite[\S 3]{Zu19} and use tilde notation to indicate modifications.

Take $s=11$ instead of $s=13$. Let $(\eta_0,\eta_1,\ldots,\eta_{s})$ be a collection of integral parameters satisfying
\[ 0<\eta_{j}<\frac{1}{2} \eta_{0} \text {~for~} j=1, \ldots, s, \quad \text { and } \quad \eta_{1}+\eta_{2}+\cdots+\eta_{s} \leqslant \frac{s-1}{2} \eta_{0}. \]
We assign for each positive even integer $n$ the collection
\[ h_{0}=\eta_{0} n+1, \quad h_{j}=\eta_{j} n+\frac{1}{2} \text {~for~} j=1, \ldots, s. \]

Let
\[ \widetilde{R}_{n}(t)=\widetilde{R}_{n, \mathbf{\eta}}(t)=\widetilde{\gamma}_{n} \cdot\left(2 t+h_{0}\right) \frac{(t+1)_{h_{0}-1}}{\prod_{j=1}^{s}\left(t+h_{j}\right)_{1+h_{0}-2 h_{j}}}, \]
where the normalization factor $\widetilde{\gamma}_{n}$ is different from $\gamma_{n}$ in \cite[\S 3]{Zu19}:
\[ \widetilde{\gamma}_{n} =4^{h_{0}-1} \frac{\prod_{j=1}^{s}\left(h_{0}-2 h_{j}\right) !}{n !^{\eta_0}} .\]

Define the sum
\[ \widetilde{r}_{n}=\sum_{\nu=0}^{\infty} (-1)^{\nu} \widetilde{R}_{n}(\nu). \]
Let
\[ N=\min _{1 \leqslant j \leqslant s}\left\{h_{j}-\frac{1}{2}\right\} \quad \text { and } \quad \widetilde{M} = h_0 - 2N-1.  \]
Note that $\widetilde{M}$ is different from $M=\max \left\{h_{0}-2 N-1, h_{1}-\frac{1}{2}\right\}$ (in the third line before Lemma 4 of \cite[\S 3]{Zu19}.)

The different normalization factor $\widetilde{\gamma}_{n}$ induces the obvious change in \cite[\S 3, Lemma 2]{Zu19}:
\[  \lim _{n \rightarrow \infty} \widetilde{r}_{n}^{1 / n}=\left(4 \eta_{0}\right)^{\eta_{0}} \cdot \max _{t \in[0,1]^{s}} \frac{\prod_{j=1}^{s} t_{j}^{\eta_{j}}\left(1-t_{j}\right)^{\eta_{0}-2 \eta_{j}}}{\left(1+t_{1} t_{2} \cdots t_{s}\right)^{\eta_{0}}},\]
and $\varphi(x,y)$ in \cite[\S 3, Lemma 4]{Zu19} is replaced by
\begin{align*}
\widetilde{\varphi}(x, y) =&~\left\lfloor 2\left(\eta_{0} x-y\right)\right\rfloor+\lfloor 2 y\rfloor-\left\lfloor\eta_{0} x-y\right\rfloor-\lfloor y\rfloor-\eta_{0}\left\lfloor x\right\rfloor \\
&+\sum_{j=1}^{s}\left(\left\lfloor\left(\eta_{0}-2 \eta_{j}\right) x\right\rfloor-\left\lfloor y-\eta_{j} x\right\rfloor-\left\lfloor\left(\eta_{0}-\eta_{j}\right) x-y\right\rfloor\right).
\end{align*}
Accordingly, $\varphi_0(x)$ and $\Phi_n$ in \cite[\S 3, Lemma 4]{Zu19} are replaced by $\widetilde{\varphi}_0(x) = \min _{0 \leqslant y<1} \widetilde{\varphi}(x, y)$ and
\[ \widetilde{\Phi}_{n}=\prod_{\sqrt{2 h_{0}}<p \leqslant \widetilde{M}} p^{\widetilde{\varphi}_{0}(n / p)}, \]
respectively.

Suppose that the partial-fraction decomposition of $\widetilde{R}_{n}(t)$ is
\[ \widetilde{R}_{n}(t) = \sum_{i=1}^{s} \sum_{k=N}^{h_{0}-N-1} \frac{\widetilde{a}_{i, k}}{\left(t+k+\frac{1}{2}\right)^{i}}. \]
As usual, we denote by $d_{\widetilde{M}}$ the least common multiple of $1,2,\ldots,\widetilde{M}$. Most importantly, Lemma $4$ of \cite[\S 3]{Zu19} can be interpreted as
\begin{equation}\label{Lemma4_for_beta}
\widetilde{\Phi}_{n}^{-1} d_{\widetilde{M}}^{s-i} \widetilde{a}_{i,k} \in \mathbb{Z}
\end{equation}
for any $i=1, \ldots, s$ and $N \leqslant k \leqslant h_{0}-N-1$. (For the parameters in \cite[\S 3]{Zu19}, we have $\widetilde{M} = M$. But for our parameters below, we have $\widetilde{M} < M$. This is the reason that we take the different normalization factor $\widetilde{\gamma}_{n}$.) To prove \eqref{Lemma4_for_beta}, we need to replace the product of three integer-valued polynomials
\[ \frac{4^{h_{1}^{*}}\left(t+\frac{1}{2}\right)_{h_{1}^{*}}}{h_{1}^{*} !}, \frac{4^{h_{0}-2 h_{1}}\left(t+h_{1}^{*}+\frac{1}{2}\right)_{h_{0}-2 h_{1}}}{\left(h_{0}-2 h_{1}\right) !}, \frac{4^{h_{1}^{*}}\left(t+h_{0}-h_{1}^{*}-\frac{1}{2}\right)_{h_{1}^{*}}}{h_{1}^{*} !} \]
(where $h_{1}^{*}=h_{1}-\frac{1}{2}=\eta_{1} n$) in the proof of Lemma $4$ of \cite[\S 3]{Zu19} by the product of the following $\eta_{0}$ polynomials:
\[ \frac{4^{n}\left(t+\frac{1}{2}\right)_{n}}{n !}, \frac{4^{n}\left(t+n+\frac{1}{2}\right)_{n}}{n !}, \frac{4^{n}\left(t+2 n+\frac{1}{2}\right)_{n}}{n !}, \ldots, \frac{4^{n}\left(t+h_{0}-n-\frac{1}{2}\right)_{n}}{n !}. \]
(See also Remark \ref{rmk} for some explanation about such replacements.)

Finally, by taking the parameters as
\[ \left(\eta_{0}, \eta_{1}, \ldots, \eta_{11}\right)=(94,32,32,32,32,33,34,35,36,37,38,39), \]
we obtain that
\[ \lim _{n \rightarrow +\infty} \widetilde{r}_{n}^{1 / n}=\exp (118.624566 \ldots) \]
and
\[\lim _{n \rightarrow +\infty}\left(\widetilde{\Phi}_{n}^{-1} d_{\widetilde{M}}^{11}\right)^{1 / n}=\exp (-118.836817 \ldots).\]
(The MATLAB code \textsf{beta10.m} can be downloaded at \url{https://github.com/lzhou-xyz/zeta35/}.) This means that the positive linear forms
\[
\widetilde{\Phi}_{n}^{-1} d_{\widetilde{M}}^{11} \widetilde{r}_{n} \in \mathbb{Z} \beta(2)+\mathbb{Z} \beta(4)+\cdots+\mathbb{Z} \beta(10)+\mathbb{Z}
\]
tend to $0$ as $n \rightarrow +\infty$. Thus, one of the $\beta$-values considered is irrational. The proof of Theorem \ref{secondtheorem} is complete.

\end{document}